\documentclass[11pt]{article}

\usepackage{amsmath, amsfonts, amssymb}
\usepackage{mathrsfs}
\usepackage{theorem}

\usepackage{bm}
\newtheorem{thm}{Theorem}[section]
\newtheorem{prop}[thm]{Proposition}
\newtheorem{lem}[thm]{Lemma}

{\theorembodyfont{\rmfamily} \newtheorem{rem}[thm]{Remark}}
\newcommand{\ra}{\rightarrow}
\newcommand{\dis}{\displaystyle}

\def\R{\mathbb R}
\def\Z{\mathbb Z}
\def\N{\mathbb N}

\def\d{\text{\rm{d}}}
\def\E{\mathbb E}
\def\p{\mathbb P}
\def\q{\mathbb Q}

\def\veps{\varepsilon}

\def\cps{C_x([0,\infty))}

\newcommand{\fin}{\hspace*{\fill}\rule{0.3em}{1ex}}
\newenvironment{proof}{{\bf \noindent Proof.}}{\fin}
\numberwithin{equation}{section}

\begin{document}

\title{Boundary crossing probabilities for diffusions with piecewise linear drifts}
\author{Jinghai Shao${}^{a}$ and Liqun Wang${}^b$\thanks{Corresponding author: Department of Statistics, University of Manitoba, Winnipeg, Manitoba, Canada R3T 2N2. Email: liqun.wang@umanitoba.ca. Tel.: 204-474-6270, Fax: 204-474-7621.}\\[0.6 cm]
{\small a:
School of Mathematical Sciences, Beijing Normal University, 100875 Beijing, China}\\
{\small b: Department of Statistics, University of Manitoba, Winnipeg, Manitoba, Canada R3T 2N2}\\
{\small School of Science, Beijing Jiaotong University, China}}
%\date{}
\maketitle

\begin{abstract}
We propose an approach to approximate the boundary crossing probabilities for general one-dimensional diffusion processes, and derive the convergence rate for this approximation scheme. There results are based on the explicit expression of the Laplace transforms of the first passage densities for diffusions with piecewise linear drifts.
%
%We study boundary crossing probabilities of general one-dimensional diffusion processes with respect to constant boundaries. We first prove the existence of the first passage density for diffusion processes with piecewise linear drift coefficients. Then we show that the Laplace transforms of these first passage densities have explicit forms. We also propose an approach to approximate the boundary crossing probabilities for general diffusion processes and derive the convergence rate for the approximation scheme.
The proposed method is applied to a reliability problem where the standard degradation model based on Wiener process is extended to diffusion processes with piecewise linear drifts.
\end{abstract}
%\end{minipage}
%\end{center}

\noindent\textbf{Keywords:} Boundary crossing probability; Diffusion processes; First passage density; First passage time; Piecewise linear drifts.

\paragraph{AMS 2010 Subject Classification:}
Primary 60J65, 60J75, Secondary 60J60, 60J70

\section{Introduction}

Boundary crossing probabilities for stochastic processes play an important role in many research areas. For example, in business and industry
a company's financial status can be represented by its total asset and debt at any given time, which can be regarded as two stochastic
processes. Then the company goes bankrupt when its asset process falls below its debt process for the first time. Therefore the company's
default risk can be described by the probability that the difference of the asset and debt processes reaches a certain threshold. In finance,
the arbitrage-free price of a barrier option is given by the expectation of the payoff function of strike price and the first hitting time of
the underlying asset price to the barrier. In neuroscience, a popular integrate-and-fire model assumes that a neuron fires a spike when the
membrane potential reaches a threshold. In engineering, the system reliability can be represented by the probability that a damage process
crossing a threshold. In epidemiology, if the spread of an infectious disease is modelled by a stochastic process, then it is of interest and
importance to know the probability that this process reaches a certain threshold within a certain time limit. The first passage time also
arises in many other disciplines such as biology, chemistry, ecology, economics, environmental science, genetics, physics as well as
statistics. Some references for applications can be found in \cite{WP07}. More recent literature includes \cite{TF}, \cite{TM} and \cite{MTKP}.

Despite their importance and wide applications, the calculation of boundary crossing probabilities is a difficult task. It is well-known that explicit formulas exist only for a few special processes and boundaries. For more general processes or boundaries, one has to rely on certain numerical approximation schemes. Wang and P\"otzelberger (1997) \cite{WP97} derived an explicit formula for the probability of Brownian motion crossing a piecewise linear boundary, and used this formula to approximate the crossing probabilities of Brownian motion to general nonlinear boundaries. The numerical computation can be done by Monte Carlo integration and the accuracy of the approximation can be computed automatically. This approach was extended to two-sided boundary crossing problems by Novikov et al (1999) \cite{NFK}, P\"otzelberger and Wang (2001) \cite{PW} and Borovkov and Novikov (2005) \cite{BN}.

Although the boundary crossing problem has been intensively studied for decades, most works concentrate on Brownian motion or some special processes such as Ornstein-Uhlenbeck processes.
The calculation of the boundary crossing probabilities for general diffusion processes remains a challenging and interesting problem. Wang and P\"otzelberger (2007) \cite{WP07} proposed a transformation approach to compute the boundary crossing probabilities for a class of processes which can be expressed as piecewise monotone functionals of standard Brownian motion. However, from both theoretical and practical points of view this class of processes is very limited. Moreover, it is difficult to generalize this transformation method to deal with other diffusion processes that are not functionals of Brownian motion.

In this paper, we propose a novel approach to the boundary crossing problem for general diffusion processes. Specifically, we first derive an explicit formula for the Laplace transform of boundary crossing density for diffusion processes with piecewise linear drift. Then we show that a general diffusion process can be approximated by a sequence of diffusion processes with piecewise linear drifts, so that the corresponding crossing probabilities can be obtained. The proposed method is fairly general so that the obtained results cover a large class of diffusion processes.

To simplify notation, in this paper we consider the diffusion processes satisfying the stochastic differential equation (SDE)
\begin{equation}\label{1.1}
\d X_t=\mu(X_t)\d t+\d W_t,\quad X_0=x.
\end{equation}
There is no loss of generality in the sense that any diffusion process satisfying
\begin{equation}\label{1.2}
\d X_t=\mu(X_t)\d t+\sigma(X_t)\d W_t
\end{equation}
with differentiable and non-zero $\sigma(x)$ can be transformed into one with unit diffusion coefficient through the transformation
$$F(y)=\int_{y_0}^y\frac1{\sigma(u)}\d u$$
for some $y_0$. Indeed, it is easy to verify by Ito's formula that the process $(Y_t)_{t\geq 0}$ with $Y_t:=F(X_t)$ $(t\geq 0)$ satisfies SDE (\ref{1.1}) with drift coefficient
$\dis \mu(x)=(\mu/\sigma-\sigma'/2)\circ F^{-1}(x)$ (see e.g., \cite{DB}).

For any process $(X_t)_{t\geq 0}$ satisfying (\ref{1.1}) we consider the first crossing time over a constant boundary $c\in\R$:
\begin{equation}\label{1.4}
\tau_c=\inf\{t>0;\ X_t\geq c\}.
\end{equation}

In Section 2,  we first study the existence of the first passage densities for diffusion processes, then apply this result to diffusion processes with piecewise linear drifts, and obtain the explicit formula of its corresponding Laplace transform. In Section 3, we propose a method to approximate the boundary crossing probabilities for a general diffusion process through its approximation by a sequence of diffusion processes with piecewise linear drifts. The approximation rate is also derived. In Section 4, the proposed method is applied to an engineering reliability problem to establish a new degradation model, which generalizes the known degradation models based on Wiener process. Finally, all theoretical proofs are contained in the Appendix.

\section{The processes with piecewise linear drift}

In this section, we first establish a general result on the existence of the first passage density. Let $(X_t)_{t\geq 0}$ be a process satisfying SDE (\ref{1.1}). Assume that $\mu$ satisfies the following conditions:
\begin{itemize}
  \item[$\mathrm{(H1)}$] there exists a constant $K_0$ such that $|\mu(x)|\leq K_0(1+|x|)$, for all $x\in\R$;
  \item[$\mathrm{(H2)}$] there exists a constant $K_1$ such that
 \[|\mu(x)-\mu(y)|\leq K_1|x-y|,\quad \forall\ x,\,y\in\R.\]
\end{itemize}
Then it is well known that SDE (\ref{1.1}) admits a unique nonexplosive solution  (see, e.g., \cite{IW}).

%Given a probability space $(\Omega, \mathcal F, \p)$, let $x\in \R$ and $(W^x_t)_{t\geq 0}$ be a Brownian motion starting at $W_0^x=x$.  For
%the simplicity of notation, subsequently we will write $W_t^x=W_t$.
Let $C_x([0,\infty))$ be the space of continuous paths starting at $x$, i.e.
$$C_x([0,\infty))=\big\{\omega:[0,\infty)\ra \R\ \text{continuous};\, \omega_0=x\big\}.$$
Let $\q_x$ be the Wiener measure on $C_x([0,\infty))$, then the process $W_t(\omega):=\omega_t$, $t\geq 0$ is a Brownian motion under $\q_x$.  Define $\mathcal F_t=\sigma(W_s; 0\leq s\leq t)=\sigma(\omega_s;\ 0\leq s\leq t, \omega\in \cps)$.
The following theorem establishes the existence of the first passage density for $(X_t)_{t\geq 0}$, the proof of which is given in the Appendix by using the Besicovitch derivation theorem.

\begin{thm}[Existence of density]\label{t2.3}
Let $(X_t)_{t\geq 0}$ be defined by (\ref{1.1}) with $\mu(\cdot)$ satisfying (H1), (H2) and   $M:=\inf_{y\in\R}\big\{\mu(y)^2+\frac13 \mu_{-}'(y)\big\}>-\infty $, where
$\mu_{-}'(y)=\liminf_{z\ra y}\frac{\mu(z)-\mu(y)}{z-y}$. Then the distribution of the first crossing time $\tau_c$ defined in (\ref{1.4}) is absolutely continuous with respect to Lebesgue measure on $[0,\infty)$ and hence its density $f_c(t,x)$ exists almost everywhere on $[0,\infty)$ for each $x<c$. Moreover, the density $f_c(t,x)$ satisfies
\begin{equation}\label{2.7}
f_c(t,x)\leq \frac{c-x}{\sqrt{2\pi} t^{3/2}}e^{G(c)-G(x)-\frac{3Mt}{2}} e^{-\frac{(c-x)^2}{2t}},
\end{equation}
where $G(y)=\int_{y_0}^y\mu(z)\d z$ for some fixed $y_0$.%$x<c$ is the initial state of the process $(X_t)_{t\geq 0}$,
\end{thm}

Now we consider the process $(X_t)_{t\geq 0}$ satisfying (\ref{1.1}) with continuous and piecewise
linear drift $\mu(x)$ in the form
\begin{equation}\label{pw linear}
\mu(x)=\!\sum_{i=-\infty}^{\infty} \mathbf{1}_{[x_{i-1},x_i]}(x)(a_i x+b_i),
\end{equation}
where $\sup\{|a_i|\!+\!|b_i|;\ i=0,\pm 1,\pm 2,\cdots\}<\infty$ and $ \{\ldots<x_{-m}<\ldots<x_{-1}<x_0<x_1<\ldots<x_m<\ldots\}$ is a partition of $\R$.
According to Theorem \ref{t2.3}, it is clear that the  density $f_c(t,x)$ of $\tau_c$ exists almost everywhere for $(X_t)_{t\geq 0}$ with piecewise linear drift.
Next, we shall use the method of Darling and Siegert \cite{DS} to obtain the concrete Laplace tranform of the distribution of $\tau_c$. Let
\begin{gather*}
\hat f_c(x|\lambda)=\int_0^\infty e^{-\lambda t} f_c(t,x)\d t.
\end{gather*}
It is shown in \cite{DS} that if $f_c$ exists, then it is given by
\begin{equation}\label{2.10}
 \hat f_c(x|\lambda)=\frac{u(x)}{u(c)},\quad x<c,
\end{equation}
where the function $u$ is any solution of the ordinary differential equation
 \begin{equation}\label{2.11}
 \frac 12 \frac{\d^2 y}{\d x^2}+\mu(x)\frac{\d y}{\d x} -\lambda y=0.
 \end{equation}
In general, it is difficult to find the explicit solutions of the above differential equation. However, we show that it is possible to find such explicit solutions when $\mu(\cdot) $ has the form of (\ref{pw linear}). Thus, we obtain the following main theorem.

\begin{thm}\label{main}
Let $(X_t)_{t\geq 0}$ be the solution of SDE (\ref{1.1}) with $\mu$ satisfying (\ref{pw linear}). Then for $x<c$, the Laplace transform of the first passage density $f_c(t,x)$ is given by
\begin{equation}\label{2.12}
\hat f_c(x|\lambda)= \frac{u(x)}{u(c)},
\end{equation}
with
\begin{equation}\label{2.13}
u(x)=\sum_{i=-\infty}^{\infty} \Big[J_i\big(-\frac{\lambda}{2 a_i}, \frac 12; -a_i \big(x+\frac{b_i}{a_i}\big)^2\big)\mathbf 1_{a_i\neq
0}+\tilde J_i(b_i;x)\mathbf 1_{a_i=0}\Big]\mathbf 1_{[x_{i-1},x_i]}(x),
\end{equation}
where
\begin{align*}
J_i(a,\frac 12;x)&=C_{1,i} \Psi(a,\frac 12;x)+C_{2,i} x^{1/2}\Psi(a+\frac 12,\frac 32;x),\\
\Psi(a, \frac 12;x)&= 1+\sum_{k=1}^\infty \frac{(a)_k}{(\frac 12)_k} \frac{x^k}{k!},\quad (a)_k:= a(a+1)\ldots (a+k-1),\\
\tilde J(b_i;x)&=
\begin{cases} e^{-b_i x}\big[C_{1,i} e^{\Delta_i x/2}+C_{2,i} e^{-\Delta_i x/2}\big], &\text{if}\
\Delta_i^2=4b_i^2+8\lambda>0,\\
e^{-b_i x}\big[C_{1,i} \sin(\Delta_i x/2)+C_{2,i} \cos(\Delta_i x/2)\big], &\text{if}\
\Delta_i^2=-4b_i^2-8\lambda>0,\\
e^{-b_ix}(C_{1,i} x+C_{2,i}), &\text{if}\ \Delta_i^2=4b_i^2+8\lambda=0.
\end{cases}
\end{align*}
Here the constants $C_{1,i}, \,C_{2,i}$ are chosen such that $u$ is differentiable and they can be determined recursively as: given $C_{1,i}, C_{2,i}$ for all $i\leq k$, $C_{1,k+1},\, C_{2,k+1}$ are chosen such that $\lim_{x\uparrow x_k}u(x)=\lim_{x\downarrow x_k} u(x)$ and $\lim_{x\uparrow x_k}\frac{u(x)-u(x_k)}{x-x_k}=\lim_{x\downarrow x_k} \frac{u(x)-u(x_k)}{x-x_k}$. Furthermore, $u$ is uniquely determined up to a multiplicative constant, which will not change the value of $\hat f(x|\lambda)$.
\end{thm}

\begin{rem}
It is worthwhile to include some asymptotic properties of $\Psi(a,b;x)$ here:
\begin{align*}\Psi(a,b;x)&=\frac{\Gamma(b)}{\Gamma(a)}e^xx^{a-b}\Big[1+O\Big(\frac1{|x|}\Big)\Big] \quad \text{if}\ x\ra +\infty,\\
\Psi(a,b;x)&=\frac{\Gamma(b)}{\Gamma(b-a)}(-x)^{-a}\Big[1+O\Big(\frac1{|x|}\Big)\Big]\quad \text{if}\ x\ra -\infty,
\end{align*}
where $\Gamma(z)=\int_0^\infty t^{z-1}e^{-t}\d t$ is the gamma function.
\end{rem}

\begin{rem}
Our approach in this section can be used to deal with the first passage time of two-sided constant boundaries. First, similar to the treatment of Theorem \ref{t2.3}, we can prove that the first passage density for two-sided constant boundaries also exists. Then using \cite[Theorem 3.2]{DS} we can obtain the explicit expression of the Laplace transform of the corresponding first passage time density. Inevitably, the expressions in this case will be more complicated.
\end{rem}

\begin{rem}
Given the Laplace transform $\hat f_c(x|\lambda)$ in (\ref{2.12}), the first passage density can be obtained by Laplace inversion. Thanks to the recent development in numerical techniques and the availability of high-speed computers, efficient numerical algorithms are available for numerical inversion of Laplace transforms (see, e.g., \cite{Kou2003}). Traditionally, the
Laplace transform has been a power tool in applied mathematics and well-studied in the literature. Many methods, such as complex analysis, residue
computations and Fourier integral inversion theorem have been developed to calculate the inversion of the Laplace transform. For results using elementary analysis, see, e.g., \cite{Br}, \cite{Hsu} and \cite{Po}.
\end{rem}

\section{Approximation for general diffusion processes}

In this section we consider the boundary crossing problem for general diffusion processes. The basic idea of our approach is that any diffusion process with sufficiently smooth drift function can be approximated by a diffusion with piecewise linear drift, so that its boundary crossing probabilities can be approximated by the corresponding probabilities. The key issue in this approach is to estimate the accuracy of this approximation.

To this end we first establish a general rule to control the approximation accuracy of the boundary crossing probabilities based on the approximation of drift functions. Let $(X_t)_{t\geq 0}$ be a diffusion process satisfying
\begin{equation}\label{nhp}
\d X_t=\mu(X_t)\d t+\d W_t,\quad X_0=x,
\end{equation}
where $\mu(\cdot)$ is lower bounded and satisfies (H1), (H2). Let $\mu_{l}=\inf_{y\in \R} \mu(y) > -\infty$.
For each $\veps>0$, let $(X_t^\veps)_{t\geq 0}$ be  a diffusion process  satisfying the SDE
\begin{equation}\label{anhp}\d X_t^\veps=\mu_\veps(X_t^\veps)\d t+\d W_t,\quad X_0^\veps=x.
\end{equation}
Assume that
\begin{itemize}
\item[(H3)] there exists a constant $K_2$ such that $|\mu_\veps(y)-\mu_\veps(z)|\leq K_2|y-z|$, \
    $\forall\, y,\, z\in \R$ and $\veps>0$;
\item[(H4)] $\dis \sup_{y\in \R}|\mu(y)-\mu_\veps(y)|\leq \veps$.
\end{itemize}
Further, for any $c>x$, let the first passage time of the process $(X_t^\veps)_{t\geq 0}$ over boundary $c$ be defined as $\tau_c^\veps=\inf\{t>0;\ X_t^\veps\geq c\}$.

\begin{prop}\label{t3.1}
Let $(X_t)_{t\geq 0}$ and $(X_t^\veps)_{t\geq 0}$ be defined as in (\ref{nhp}) and (\ref{anhp}) with $\mu_{l}=\inf_{y\in \R} \mu(y)>-\infty$. Then under (H3) and (H4), for every $T>0$, it
holds
\begin{equation}\label{3.1}
\begin{split}
&\big|\p(\tau_c>T)-\p(\tau_c^\veps>T)\big|\\
&\leq 2Te^{3K_2T/2}e^{G(c)-G(x)}\Big[\int_0^T\!\!\! \frac{c\!-\!x}{\sqrt{2\pi} s^{3/2}}
e^{-\frac{(c-x)^2}{2s}}\Big(|\mu_{l}|\!+\frac{1}{\sqrt{2\pi(T-s)}}\Big)\d s\Big]\veps+o(\veps),
\end{split}
\end{equation}
where $G(y)=\int_{y_0}^y \mu(z)\d z$ for some fixed $y_0$.
\end{prop}

\begin{rem}
The similar result was proved by Downes and Borovkov \cite[Theorem 4.1]{DB} for the case where both diffusion processes have differentiable drift coefficients. Here we generalize their result to the case of non-differentiable drift coefficent.
\end{rem}

Now let $(X_t)_{t\geq 0}$ be the solution of SDE (\ref{nhp}) and assume that there exist positive constants $M_1,\,M_2$ such that
\begin{equation}\label{3.5}
\sup_{y\in \R} |\mu(y)|\leq M_1,\quad \sup_{x,y\in \R} |\mu(x)-\mu(y)|\leq M_2|x-y|.
\end{equation}
For each $n\in \N$, we partition the real line $\R$ with sub-intervals with endpoints $\{0,\pm \frac 1n,\pm\frac 2n,\ldots\}$.
Let
$\mu_n(y)$ be the piecewise linear function taking value $\mu(x_i)$ at $x_i=\frac in$ $(i\in \Z)$. Thus, for each $i\in \Z$,
\begin{equation}\label{3.6}
\begin{split}
\max_{y\in [x_i,x_{i+1}]} |\mu(y)-\mu_n(y)|&\leq \max_{y\in [x_i,x_{i+1}]}\big\{\max\{|\mu(y)-\mu(x_i)|,|\mu(y)-\mu(x_{i+1})|\}\big\}\\
&\leq  M_2/n.
\end{split}
\end{equation}
Further, let $(X_t^{(n)})_{t\geq 0}$ be the solution of the SDE
\begin{equation}\label{3.7}
\d X_t^{(n)}=\mu_n(X_t^{(n)})\d t+\d W_t,\quad X_0^{(n)}=x.
\end{equation}
Then for each $n\in \N$, $(X_t^{(n)})_{t\geq 0}$ is a diffusion process with piecewise linear drift and therefore, for $c>x$, the distribution of its first passage time $\tau_c^{(n)}:=\inf\{t>0;\ X_t^{(n)}\geq c\}$ is given in Theorem \ref{main} in terms of its Laplace transform.
Furthermore, by the definition of $\mu_{n}(y)$ it is easy to see that
\begin{equation}\label{3.9}
\sup_{y\in \R} |\mu_{n}(y)|\leq M_1,\quad |\mu_{n}(y)-\mu_n(z)|\leq M_2|y-z|,\quad y,\,z\in \R,
\end{equation}
which implies $G(c)-G(x)\leq M_1(c-x)$. Therefore by Proposition \ref{t3.1} and (\ref{3.6}) we obtain the following result.

\begin{thm}\label{t3.2}
Let $(X_t)_{t\geq 0}$ and $(X_t^{(n)})_{t\geq 0}$ be the diffusion processes satisfying (\ref{nhp}) and (\ref{3.7}) respectively. Then under the condition (\ref{3.5}), for $c>x$, $T>0$, it holds
\begin{equation}\label{3.8}
\begin{split}
&\big|\p(\tau_c>T)-\p(\tau_c^{(n)}>T)\big|  \leq C(T,x,c)\frac 1n+o\Big(\frac 1n\Big),
\end{split}
\end{equation}
where $\displaystyle  C(T,x,c)=2TM_2e^{3M_2T/2}e^{M_1(c-x)}\Big[\int_0^T\!\!\! \frac{c\!-\!x}{\sqrt{2\pi s^3}} e^{-\frac{(c-x)^2}{2s}}\Big(M_1\!+\!\frac
1{\sqrt{2\pi(T-s)}}\Big)\d s\Big]$.
\end{thm}

\section{An application to reliability theory}

In this section we apply our method to a modelling problem in reliability theory and thereby demonstrate how our methods can be used in practice. Generally speaking, the reliability declines as the underlying system degrades or deteriorates, and the system fails when the level of degradation reaches a certain threshold.
As pointed out by Singpurwalla (1995) \cite{Sing}, using stochastic processes to describe failure models has a more realistic motivation, and this approach better exploits the physics of the failure
process and offers potential for improved assessments of item survivability. Typical stochastic processes are Poisson process, Wiener process and general L\'evy process.  These models may provide a better goodness-of-fit to failure data than commonly used exponential or Weibull distributions. But the challenge of this approach is that calculating the first passage time
distribution is generally quite difficult. So only a small class of diffusion processes, such as the
Brownian motion with state-independent drifts and Ornstein-Uhlenbeck processes, has been used. For example, Wiener process degradation models
have found application in \cite{DH,DN}. Wiener processes with a non-random timescale were used as degradation models for heating cables and
Carbon-film resistors in \cite{PT,WS}.

In \cite{KL}, Kahle and Lehmann considered the model
\begin{equation}\label{Wiener}
Z_t =z_0+\sigma W_{t-t_0}+\mu(t-t_0),\ t\geq t_0,
\end{equation}
where $(W_t)_{t\geq 0}$ is a Wiener process. Moreover, Wang \cite{Wx} considered the above model with random drifts and diffusion coefficients. The
main objects in \cite{KL} and \cite{Wx} are the parameter estimation from observation of degradation or observation of failures. Now we extend the above model to the following form:
\begin{equation}\label{model}
\d X_t= \big[b_1\mathbf{1}_{\{X_t\leq \theta\}}+ (a_2X_t+b_2)\mathbf{1}_{\{X_t\geq \theta\}}\big]\d t +\d
B_t,\quad X_0=0.
\end{equation}
This model reflects the fact that the degradation process $(X_t)_{t\geq 0}$ has different kind of degradation rate according to the stage of $(X_t)_{t\geq 0}$. To simplify our discussion, we assume the threshold $\theta $ of $(X_t)_{t\geq 0}$ is given a priori.
In particular, in what follows we fix the constant $\theta=1$. Hence the question is how to estimate unknown parameters $b_1,\, a_2,\, b_2$ given the observations of the first passage time (failure time).

First, we can use the observations of the first passage time through a boundary $0<c<1$ and the method of \cite{KL} to estimate the parameter $b_1$. The main challenge is to estimate parameters $a_2,\, b_2$. Here we propose to
use the method of moments to estimate them, since the moments of the first passage time are easily calculated using (\ref{2.12}). Specifically, we first choose and fix the threshold $c>1$ for model (\ref{model}).
Note that the continuity of $\mu(\cdot)$ implies
\begin{equation}\label{coeff-cond}
b_1=a_2+b_2.
\end{equation}
According to Theorem \ref{main},
\[\hat{f}_c(0|\lambda)=\frac{u(0)}{u(c)},\]
and $u$ is given by (\ref{2.13}). In order to determine the constants $C_{1,i},\,C_{2,i}$ in $J_i$ and $\tilde J_i$, note that we have
\[
u(x)=J(-\frac{\lambda}{2 a_2}, \frac 12, -a_2(x+\frac{b_2}{a_2})^2),\ \text{ for $x\geq 1$}.
\]
Further, using the conditions
$\lim_{x\uparrow 1} u(x)=\lim_{x\downarrow 1} u(x)$ and $u'_+(1)=u'_{-}(1)$, we obtain
\[
u(x)=C_1(\lambda) e^{-b_1x+\frac{\Delta_1 x}{2}}+C_2(\lambda)e^{-b_1x-\frac{\Delta_1 x}{2}},\quad  \text{for}\quad x\leq 1,
\]
where $\Delta_1=\sqrt{4b_1^2+8\lambda}$ and
\begin{align*}
C_1(\lambda)&= e^{b_1+\frac{\Delta_1 }{2}}\Big[\big(\frac 12+\frac{b_1}{\Delta_1}\big)\Psi\big(-\frac{\lambda}{2 a_2}, \frac 12;
-b_1\big)+\frac 1{\Delta_1}\Psi_x'\big(-\frac{\lambda}{2 a_2}, \frac 12; -b_1\big)\Big],\\
C_2(\lambda)&= \frac{1}{\Delta_1}e^{b_1+\frac{\Delta_1}{2}}\Big[(\frac{\Delta_1}{2}-b_1)\Psi\big(-\frac{\lambda}{2 a_2}, \frac 12;
-b_1\big)-\Psi_x'\big(-\frac{\lambda}{2 a_2}, \frac 12; -b_1\big)\Big].
\end{align*}
Thus, we obtain the explicit expression
\[\hat f(0|\lambda)=\frac{C_1(\lambda)+C_2(\lambda)}{C_1(\lambda)e^{-b_1c+\frac{c\Delta_1}{2}}+C_2(\lambda)e^{-b_1c-\frac{c\Delta_1}{2}}}.\]
Moreover, some elementary calculations yield
\begin{gather*}
\Psi(a, \frac 12; x)=\sum_{k=1}^\infty \frac{(a)_k}{(\frac12)_k} \frac{x^k}{k!},\quad
\Psi_a'(0,\frac 12; x)=\sum_{k=1}^\infty \frac{x^k}{k(\frac 12)_k},\\ \Psi_x'(0,\frac 12; x)=0, \quad
\Psi_{xa}''(0,\frac 12; x):=\frac{\partial^2 \Psi}{\partial x\partial a}\big(a,\frac12;x\big)=\sum_{k=1}^\infty \frac{x^{k-1}}{\big(\frac12)_k},\\
C_1(0)=e^{b_1+|b_1|}\big(\frac 12+\frac{b_1}{2|b_1|}\big),\quad C_2(0)=\frac{1}{2|b_1|}e^{b_1+|b_1|}\big(|b_1|-b_1\big),
\end{gather*}
and
\begin{align*}
C_1'(0)&=\big(\frac 1{2|b_1|}+\frac 1{2b_1}-\frac{1}{2b_1|b_1|}\big) e^{b_1+|b_1|}-\frac{1}{2a_2}e^{b_1+|b_1|}\big(\frac 12+\frac{b_1}{2|b_1|}\big)\Psi_a'(0,\frac 12;-b_1)\\ &\quad-\frac{1}{2a_2}\frac{1}{2|b_1|}e^{b_1+|b_1|}\Psi_{xa}''(0,\frac 12;-b_1),\\
C_2'(0)&=e^{b_1+|b_1|}\big(\frac 12+\frac 1{2|b_1|}-\frac 1{2b_1}+\frac{1}{2b_1|b_1|}\big)-\frac 1{2a_2}e^{b_1+|b_1|}\big(\frac 12-\frac{b_1}{2|b_1|}\big)\Psi_a'(0,\frac 12;-b_1)\\&\quad-\frac{1}{2a_2}e^{b_1+|b_1|}\Psi_{xa}''(0,\frac 12;-b_1).
\end{align*}
Using these quantities we obtain the expectation
\begin{align*}
-\E_0[\tau_c]&=\frac{d}{d \lambda}\hat f(0|\lambda)\big|_{\lambda=0}
=\frac{C_1'(0)+C_2'(0)}{C_1(0)e^{-b_1c+c|b_1|}+C_2(0)e^{-b_1c-c|b_1|}}\\
&-\frac{\big(C_1(0)+C_2(0)\big)\Big(\big(C_1'(0)+\frac{c}{|b_1|}C_1(0)\big)e^{-b_1c+c|b_1|}+
\big(C_2'(0)-\frac{c}{|b_1|}C_2(0)\big)e^{-b_1c-c|b_1|}\Big)}{\Big(C_1(0)e^{-b_1c+c|b_1|}+C_2(0)e^{-b_1c-c|b_1|}\Big)^2}.
\end{align*}
That is, as a function of the parameter $a_2$, $\E[\tau_c]$ is in the form
\[\E[\tau_c]=\frac{\xi}{a_2}+\eta, \]
where $\xi$ and $\eta$ are constants depending only on $b_1$ and $c$. Therefore $a_2$, and hence $b_2$, can be estimated by the method of moment using equation (\ref{coeff-cond}).

\section{Appendix}

We use the idea of \cite{DB} to prove Theorem \ref{t2.3}. First we establish a lemma used below. Let $p(s,x,y)$ and $q(s,x,y)$ denote the transition densities of $(X_s)_{s\geq 0}$ and Brownian motion respectively. Let $\p_x$ denote the law of $(X_s)_{s\geq 0}$ starting at $X_0=x$ in $C_x([0,\infty))$, and  $\p_x^z$ the law of $(X_s)_{0\leq s\leq t}$ starting at $X_0=x$ and conditioning on $X_t=z$ in $C_x([0,\infty))$. Similarly,  define $\q_x^z$ be the corresponding probability measures for the Brownian motion $(W_s)_{s\geq 0}$.
This lemma is a generalization of \cite[Theorem 3.1]{DB} to the case where $\mu(\cdot)$ is not  necessarily everywhere differentiable.

\begin{lem}\label{t2.1}
For $t>0$ and  every $A\in \mathcal F_t$,
\begin{equation}\label{2.1}
\p_x^z(A)\leq \frac{q(t,x,z)}{p(t,x,z)} e^{G(z)-G(x)-3Mt/2}\q_x^z(A),
\end{equation}
where $G(y)=\int_{y_0}^y\mu(z)\d z$ for some fixed $y_0$, and $M=\inf_{y\in\R}\big\{\mu(y)^2+\frac13 \mu_{-}'(y)\big\}$ with
$\mu_{-}'(y)=\liminf_{z\ra y}\frac{\mu(z)-\mu(y)}{z-y}$.
\end{lem}

\begin{proof}
Since $X_t=W_t+\int_0^t \mu(X_s)\d s$, by Girsanov's theorem,
\begin{equation}\label{2.2}
\p_x(A)=\q_x(\zeta_t \mathbf{1}_{A}),\ A\in \mathcal F_t,
\end{equation} where $Q_x(\zeta_t \mathbf{1}_A)=\int \zeta_t\mathbf{1}_A\,\d Q_x$ and
$$\zeta_t=\exp\Big[\int_0^t\mu(X_s)\d W_s-\frac 12\int_0^t\mu(X_s)^2\d s\Big].$$
Conditioning on the value of corresponding processes at time $t$, due to (\ref{2.2}), we  obtain that for any Borel measurable set $B\subset\R$,
\begin{align*}
&\int_{-\infty}^{\infty} \mathbf 1_B(z)\p_x^z(A)p(t,x,z)\,\d z=\int_{-\infty}^{\infty}\mathbf 1_B(z)\p_x^z(A)\p_x(X_t\in \d z)\\
&=\int_{-\infty}^\infty \mathbf 1_B(z)\q_x^z(\zeta_t\mathbf 1_A)\q_x(W_t\in\d z)=\int_{-\infty}^\infty \mathbf 1_B(z)\q_x^z(\zeta_t \mathbf 1_A)q(t,x,z)\,\d z.
\end{align*}
This implies that
\begin{equation}\label{2.3}
\p_x^z(A)=\frac{q(t,x,z)}{p(t,x,z)}\q_x^z(\zeta_t \mathbf 1_A).
\end{equation}
Further, let $\rho(y)=1/\sqrt{2\pi}e^{-y^2/2}$ and $\mu_n(y)=\int_{-\infty}^\infty \frac{n\mu(z)}{\sqrt{2\pi}}e^{-\frac{n^2(z-y)^2}{2}}\d z$
for $n\in \N$.
Then $\mu_n(y)$ converges to $\mu(y)$ for every $y\in \R$ as $n\ra \infty$. Moreover, by the dominated convergence theorem,
$$\E\Big(\int_0^t(\mu_n(X_s)-\mu(X_s))\d W_s\Big)^2=\E\int_0^t(\mu_n(X_s)-\mu(X_s))^2\d s\longrightarrow 0,\quad \text{as}\ n\ra \infty.$$
It follows that, up to an extraction of subsequence, $\dis\int_0^t\mu_n(X_s)\d W_s\ra \int_0^t \mu(X_s)\d W_s$ almost surely.
Similarly by the dominated convergence theorem, $G_n(X_t)$ converges almost surely to $G(X_t)$ as $n\ra \infty$, where $G_n(y)=\int_{y_0}^y
\mu_n(z)\d z$.
Further, by Fatou's lemma, we have
\begin{align*}
\mu_n'(y)&=\lim_{h\ra 0}\frac{\mu_n(y+h)-\mu_n(y)}{h}\\
&=\lim_{h\ra 0}\frac 1h\int_{-\infty}^\infty \big(\mu(y+h- \frac zn)-\mu(y-\frac zn)\big)\frac 1{\sqrt{2\pi}}e^{-z^2/2}\d z\\
&\geq \int_{-\infty}^\infty \liminf_{h\ra 0}\frac{\mu(y-\frac zn+h)-\mu(y-\frac zn)}{h}\frac 1{\sqrt{2\pi}} e^{-z^2/2}\, \d z\\
&=\int_{-\infty}^\infty \mu_{-}'(y\!-\!\frac zn)\frac1{\sqrt{2\pi}}e^{-z^2/2}\,\d z.
\end{align*} Noticing that  $\mu_{-}'(y)=\liminf_{z\ra y}\frac{\mu(z)-\mu(y)}{z-y}$, by Fatou's lemma we get
$$\liminf_{n\ra \infty}\mu_n'(y)\geq\liminf_{n\ra \infty}\int_{-\infty}^\infty \mu_{-}'(y\!-\!\frac zn) \frac 1{\sqrt{2\pi}}e^{-z^2/2}\,\d
z\geq
 \mu_{-}'(y),\quad \forall\, y\in \R.$$  Hence, $\lim_{n\ra\infty}\int_0^t\mu_n'(X_s)\d s \geq \int_0^t\mu_{-}'(X_s)\d s$ almost surely.
Applying It\^o's formula to $G_n(X_t)$, we have
\begin{equation*}
\int_0^t\mu_n(X_s)\d W_s=G_n(X_t)-G_n(X_0)-\int_0^t \mu_n(X_s)\mu(X_s)\d s-\frac 12\int_0^t\mu_n'(X_s)\d s.
\end{equation*}
Passing to the limit as $n\ra \infty$ yields
\begin{equation}\label{2.4}
\int_0^t \mu(X_s)\d W_s\leq G(X_t)-G(X_0)-\int_0^t\mu(X_s)^2\d s-\frac 12 \int_0^t\mu_{-}'(X_s)\d s, \ a.s.
\end{equation}
Finally, combining (\ref{2.4}) with (\ref{2.3}) we have
$$\p_x^z(A)\leq \frac{q(t,x,z)}{p(t,x,z)} e^{G(z)-G(x)}\q_x^z\big(e^{-\frac 32\int_0^t\mu(X_s)^2\d s-\frac 12\int_0^t\mu_{-}'(X_s)\d s}\mathbf
1_A\big),$$
which yields the desired result (\ref{2.1}).
\end{proof}

\noindent\textbf{Proof of Theorem \ref{t2.3}}:\
Let $\nu_{\tau }$ denote the distribution of $\tau_c$. According to Besicovitch derivation theorem (cf. \cite[Theorem 2.22]{Am}), $\nu_{\tau }$ admits the following Radon-Nikodym decomposition with respect to the Lebesgue measure $m(\d t)$ on $(0,\infty)$,
\[\nu_{\tau }(\d t)=\nu_\tau\big|_F(\d t)+\nu_{\tau}\big|_E(\d t),\]
where
\[E=\big\{t\in (0,\infty);\ \lim_{h\downarrow 0}\frac{\nu_{\tau}((t-h,t+h))}{2h}=\infty\big\},\quad \text{and}\ \ F=(0,\infty)\backslash E.\]
Moreover, for a.e. $t\in (0,\infty)$, the limit $f_c(t,x)=\lim_{h\downarrow 0}\frac{\nu_{\tau}((t-h,t+h))}{2h}$ exists and $\nu_{\tau}\big|_{F}(\d t)=f_c(t,x)\d t$. Based on this decomposition, to show the existence of the density $f_c(t,x)$, we only need to show that
\[\limsup_{h\downarrow 0} \frac{\p(\tau_c\in (t-h,t+h))}{2h}<\infty \quad\text{for every $t>0$},\]
which yields that $E$ is an empty set, and $\nu_\tau$ is absolutely continuous  with respect to the Lebesgue measure.
In the following we shall estimate $\limsup_{h\downarrow 0} \frac{\p(\tau_c\in (t,t+h))}{h}$. The term  $\limsup_{h\downarrow 0} \frac{\p(\tau_c\in (t-h,t])}{h}$ can be estimated similarly.

We have
\begin{align*}
\p_x(\tau_c\in (t,t+h))&=\int_{-\infty}^c\p_x(\tau_c\in (t,t+h)|X_t=z)\p_x(X_t\in \d z)\\
&=\int_{-\infty}^{c-h^{1/4}}+\int_{c-h^{1/4}}^c \p_x(\tau_c\in (t,t+h)|X_t=z)\p_x(X_t\in \d z)\\
&=:\mathrm{I}_1+\mathrm{I}_2
\end{align*}
We shall estimate $\mathrm I_1$ and $\mathrm I_2$ separately. As a preparation, we estimate the probability $\p_x(\sup_{t\leq s\leq t+h} (X_s-c)\geq 0|X_t=z)$ separatively according to the distance between $z$ and $c$.

\noindent\textbf{Case 1}.\ When $z\in (-\infty, c-h^{\frac 14}]$, set $\gamma=c-h^{\frac 14}$.
If the process $(X_t)_{t\geq 0}$ goes up to cross the level $c$, it must first hit the level $\gamma$ some time during $(t,t+h)$. Thus
\begin{align*}
  &\p_x\big(\sup_{t\leq s\leq t+h}(X_s-c)\geq 0|X_t=z\big)\\
  &\leq \sup_{t\leq t'\leq t+h}\p_x\big(\sup_{t'\leq s\leq t+h} (X_s-c)\geq 0|X_{t'}=\gamma\big)\\
  &\leq \p_x\big(\sup_{t\leq s\leq t+h}(X_s-c)\geq 0|X_t=\gamma\big)=\p_\gamma\big(\sup_{0\leq s\leq h}(X_s-c)\geq 0\big).
\end{align*}
Then for $\alpha=c-h^{\frac 18}<\gamma$ we have
\begin{align*}
  &\p_{\gamma}\big(\sup_{0\leq s\leq h}(X_s-c)\geq 0\big)\\
  &=\p_\gamma\big(\sup_{0\leq s\leq h} (X_s-c)\geq 0, \inf_{0\leq s\leq h} X_s>\alpha\big)+\p_\gamma\big(\sup_{0\leq s\leq h} (X_s-c)\geq 0, \inf_{0\leq s\leq h} X_s\leq \alpha\big)
\end{align*}
Let $(Y_s)_{s\geq 0}$ be a process satisfying the SDE
\[\d Y_s=\mu_{+}\d s+\d W_s,\quad Y_0=X_0=\gamma,\ \text{where}\ \mu_{+}=\sup_{\alpha\leq y\leq c}\mu(y)<\infty.\]
Then $Y_s\geq X_s$ a.s. up to the time when $(X_s)_{s\geq 0}$ first leaves the interval $(\alpha,c)$. Consequently,
\begin{equation}
\label{ine-1}
\begin{split}
  &\p_{\gamma}\big(\sup_{0\leq s\leq h}(X_s-c)\geq 0,\ \inf_{0\leq s\leq h} X_s>\alpha\big)\\
  &\leq \p_{\gamma}\big(\sup_{0\leq s\leq h}(Y_s-c)\geq 0,\ \inf_{0\leq s\leq h} Y_s>\alpha\big)\leq \p_{\gamma}\big(\sup_{0\leq s\leq h} (Y_s-c)\geq 0\big)\\
  &=Q_0\big(\sup_{0\leq s\leq h}(W_s+\mu_{+}s+\gamma-c)\geq 0\big)\\
  &=1-\Phi\Big(\frac{-\mu_{+}h+c-\gamma}{\sqrt{h}}\Big)
  +e^{2\mu_{+}(c-\gamma)}\Phi\Big(\frac{-\mu_{+}h-c+\gamma}{\sqrt{h}}\Big).
\end{split}
\end{equation}
Let $(Z_s)_{s\geq 0}$ be a process satisfying the SDE
\[\d Z_s=\mu_{-}\d s+\d W_s,\ Z_0=X_0=\gamma,\quad \text{where}\ \mu_{-}=\inf_{\alpha\leq y\leq c}\mu(y)>-\infty.\]
Then $Z_s\leq X_s$ a.s. up to the time when $(X_s)_{s\geq 0}$ first leaves the interval $(\alpha,c)$. Consequently,
\begin{equation}\label{ine-2}
\begin{split}
  &\p_{\gamma}\big(\sup_{0\leq s\leq h} (X_s-c)\geq 0,\ \inf_{0\leq s\leq h} X_s\leq \alpha\big)\\
  &\leq \p_{\gamma}\big(\sup_{0\leq s\leq h}(X_s-c)\geq 0,\ \inf_{0\leq s\leq h}Z_s\leq \alpha\big)\\
  &\leq \p_{\gamma}\big(\inf_{0\leq s\leq h}Z_s\leq \alpha\big)=Q_0\big(\inf_{0\leq s\leq h}(W_s+\mu_{-}s+\gamma-\alpha)\leq 0\big)\\
  &=1-\Phi\Big(\frac{\mu_{-}h+\gamma-\alpha}{\sqrt{h}}\Big)
  +e^{2\mu_{-}(\alpha-\gamma)}\Phi\Big(\frac{\mu_{-}h+\alpha-\gamma}{\sqrt{h}}\Big).
\end{split}
\end{equation}

\noindent\textbf{Case 2}. When $z\in (c-h^{\frac1 4}, c)$, using the same argument as that of (\ref{ine-1}) and (\ref{ine-2}) but taking $\gamma=z$, we can get
\begin{equation}\label{ine-3}
\begin{split}
  &\p_x\big(\sup_{t\leq s\leq t+h}(X_s-c)\geq 0|X_t=z\big)\\
  &\leq 1-\Phi\Big(\frac{-\mu_{+}h+c-z}{\sqrt{h}}\Big)
  +e^{2\mu_{+}(c-z)}\Phi\Big(\frac{-\mu_{+}h-c+z}{\sqrt{h}}\Big)\\
  &\quad +1-\Phi(\frac{\mu_{-}h+z-\alpha}{\sqrt{h}}\Big)
  +e^{2\mu_{-}( \alpha-z)}\Phi\Big(\frac{\mu_{-}h+\alpha-z}{\sqrt{h}}\Big)
\end{split}
\end{equation}

Now we return to estimate $\mathrm{I}_1$ and $\mathrm{I}_2$. Due to (\ref{ine-1}) and (\ref{ine-2}) we have
\begin{align*}
\mathrm{I}_1&=\int_{\!-\infty}^{c-h^{1/4}}\!\!\!\!\p_x\big(\sup_{0\leq s\leq t}\!(X_s\!-c)\!<0|X_t\!=\!z\big)\p_x\big(\!\!\sup_{t\leq s\leq
t+h}\!\!(X_s\!-c)\!\geq 0|X_t\!=\!z\big)\p_x(X_t\!\in\! \d z)\\
&\leq \int_{\!-\infty}^{c-h^{1/4}}\!\!\!\!\p_x\big(\sup_{t\leq s\leq t+h}\!\!(X_s-c)\geq 0|X_t=z\big)\p_x(X_t\in \d z)\\
&\leq
1-\Phi\Big(\frac{-\mu_+h+h^{\frac 14}}{\sqrt{h}}\Big)
+e^{2\mu_+h^{\frac 14}}\Phi\Big(\frac{-\mu_+h-h^{\frac 14}}{\sqrt{h}}\Big)\\ &\quad +1-\Phi\Big(\frac{\mu_{-}h+h^{\frac 18}-h^{\frac 14}}{\sqrt{h}}\Big)+e^{2\mu_{-}(h^{\frac1 4}-h^{\frac 18})}\Phi\Big(\frac{\mu_{-}h+h^{\frac 14}-h^{\frac 18}}{\sqrt{h}}\Big).
\end{align*}
As
\begin{gather*}
  1-\Phi\Big(\frac{-\mu_+h+h^{\frac 14}}{\sqrt{h}}\Big)=o(h),\quad \Phi\Big(\frac{-\mu_+h-h^{\frac 14}}{\sqrt{h}}\Big)=o(h),\\
  1-\Phi\Big(\frac{\mu_{-}h+h^{\frac 18}-h^{\frac 14}}{\sqrt{h}}\Big)=o(h),\ \Phi\Big(\frac{\mu_{-}h+h^{\frac 14}-h^{\frac 18}}{\sqrt{h}}\Big)=o(h),
\end{gather*}
we obtain
\begin{equation}\label{I-1}
\limsup_{h\downarrow 0} I_1/h\leq 0.
\end{equation}

By Lemma \ref{t2.1} and (\ref{ine-3}), we get
\begin{align*}
\mathrm{I}_2&=\int_{c-h^{1/4}}^c\!\!\!\!\p_x\big(\sup_{0\leq s\leq t}\!(X_s\!-c)\!<0|X_t\!=\!z\big)\p_x\big(\!\!\sup_{t\leq s\leq
t+h}\!\!(X_s\!-c)\!\geq 0|X_t\!=\!z\big)\p_x(X_t\!\in\! \d z)\\
&\leq \int_{c-h^{1/4}}^c q(t,x,z) e^{G(z)-G(x)-\frac{3Mt}{2}}\big(1-e^{-2(c-x)(c-z)/t}\big)\\
&\qquad\qquad \cdot\Big[1-\Phi\big(\frac{-\mu_+h+c-z}{\sqrt{h}}\big)
+e^{2\mu_+(c-z)}\Phi\big(\frac{-\mu_+h-c+z}{\sqrt{h}}\big)\\
&\quad +1-\Phi(\frac{\mu_{-}h+z-\alpha}{\sqrt{h}}\Big)
  +e^{2\mu_{-}( \alpha-z)}\Phi\Big(\frac{\mu_{-}h+\alpha-z}{\sqrt{h}}\Big)\Big]\d z\\
&=\frac{2(c-x)}{t}\int_0^{h^{1/4}}\!\!\!e^{G(c-y)-G(x)-\frac{3Mt}{2}} (y+O(y^2))\Big[1-\Phi(-\mu_+h^{1/2}+y h^{-1/2}\big)\\
&\ \hspace{ 2 cm}+e^{2\mu_+y}\Phi\big(-\mu_+h^{1/2}-yh^{-1/2}\big)\Big]q(t,x,c-y)\d y+o(h).
\end{align*}
It follows that
\begin{equation}\label{2.8}
\begin{split}
&\limsup_{h\downarrow 0}\,\mathrm I_2/h
\leq K\lim_{h\downarrow 0}\frac 1h\int_0^{h^{1/4}}\!\!\!\big(y+O(y^2)\big)\Big[1-\Phi\big(-\mu_+h^{1/2}+y h^{-1/2}\big)\\
&\hspace{ 2.5cm}+e^{2\mu_+y}\Phi\big(-\mu_+h^{1/2}-yh^{-1/2}\big)\Big]q(t,x,c-y)\d y,
\end{split}
\end{equation}
where $K=\frac{2(c-x)}{t}e^{G(c)-G(x)-\frac{3Mt}{2}}$.
Note that
 \begin{align*}
& \lim_{h\downarrow 0}\frac 1h\int_0^{h^{1/4}}\!\!\! y\Big[1-\Phi\big(-\mu_+h^{1/2}+y h^{-1/2}\big)\\
& \hspace{2.5 cm}+e^{2\mu_+y}\Phi\big(-\mu_+ h^{1/2}-yh^{-1/2}\big)\Big]q(t,x,c-y)\d y\\
%&=\lim_{h\downarrow 0} \int_0^{h^{1/4}} y \frac{\d }{\d h}\Big[1-\Phi\big(-\mu_+h^{1/2}+y h^{-1/2}\big)\\
%& \hspace{2.5 cm}+e^{2\mu_+y}\Phi\big(-\mu_+ h^{1/2}-yh^{-1/2}\big)\Big]q(t,x,c-y)\d y\\
%&=\lim_{h\downarrow 0}\int_0^{h^{1/4}}\frac{y}{2\sqrt{2\pi}}\Big[e^{\frac{-(\mu_+h^{1/2}+yh^{-1/2})^2}2}\big(\mu_+h^{-1/2}+yh^{-3/2}\big)\\
%&\hspace{2 cm} +e^{-\frac{(-\mu_+h^{1/2}+yh^{-1/2})^2}2}\big(-\mu_+h^{-1/2}+yh^{-3/2}\big)\Big]q(t,x,c-y)\d y\\
%%\end{align*}
%%\begin{align*}
%&=\lim_{h\downarrow 0}\int_0^{h^{-1/4}} \frac{y}{4\pi \sqrt t}\Big[e^{-\frac{(\mu_+h^{1/2}+y)^2} 2}\big(\mu_+h^{1/2}+y\big)\\
%&\hspace{ 2 cm} +e^{-\frac{(\mu_+h^{1/2}-y)^2}2}\big(-\mu_+ h^{1/2}+y\big)\Big] e^{-\frac{(c-x-yh^{1/2})^2}{2t}} \d y\\
&=\frac{1}{2\sqrt{2\pi t}} e^{-\frac{(c-x)^2}{2t}}.
\end{align*}
Substituting this into (\ref{2.8}), we have
$$\limsup_{h\downarrow 0} I_2 / h\leq \frac{c-x}{\sqrt{2\pi} t^{3/2}}e^{G(c)-G(x)-\frac{3Mt}{2}} e^{-\frac{(c-x)^2}{2t}}.$$
Combining with the estimate of \eqref{I-1}, we finally obtain
\begin{equation}\limsup_{h\downarrow 0} \frac 1h \p_x(\tau_c\in (t,t+h))\leq \frac{c-x}{\sqrt{2\pi} t^{3/2}}e^{G(c)-G(x)-\frac{3Mt}{2}}
e^{-\frac{(c-x)^2}{2t}},
\end{equation}
which yields that $E$ is an empty set, the density $f_c(t,x)$ exists and the upper bound (\ref{2.7}) for $f_c(t,x)$ holds. We conclude the proof of Theorem \ref{t2.3}. \fin

\noindent\textbf{Proof of Theorem \ref{main}}:\
First, by \cite[2.1.2.11]{PZ} the second-order differential equations
\begin{equation*}
 \frac{\d^2 y}{\d x^2}+a\frac{\d y}{\d x}+by=0
\end{equation*}
has explicit solutions given by
\begin{equation*}
y=\begin{cases} \exp{(-\frac 12 ax)}\big[C_1\exp{(\frac 12 \Delta x)}+C_2\exp{(-\frac 12 \Delta x)}\big]\quad &\text{if}\ \Delta^2=a^2-4b>0,\\
                \exp{(-\frac 12 ax)}\big[C_1\sin(\frac 12 \Delta x)+C_2\cos(\frac 12 \Delta x)\big]\quad &\text{if}\ \Delta^2=4b-a^2>0,\\
                \exp{(-\frac 12 ax)}\big(C_1x+C_2\big)\quad &\text{if}\ a^2=4b,
\end{cases}
\end{equation*} where $C_1,\,C_2$ are constants. Similarly, by \cite[2.1.2.108]{PZ} the general solution of ODE
\begin{equation*}
\frac12 \frac{\d^2 y}{\d x^2}+(a_1x+b_1)\frac{\d y}{\d x}+b_0 y=0, \quad a_1\neq 0
\end{equation*}
is $\dis J\big(b_0,\frac 12; -a_1\big(x+\frac{b_1}{a_1}\big)^2\big)$, where
\begin{align*}
J(a,b;x)&=C_1\Psi(a,b;x)+C_2x^{1-b}\Psi(a-b+1,2-b;x),\\
\Psi(a,b;x)&=1+\sum_{k=1}^\infty\frac{(a)_k}{(b)_k}\frac{x^k}{k\,!}.
\end{align*}
Combining these two solutions together, we obtain the explicit solution (\ref{2.13}) to the differential equation (\ref{2.11}).

\noindent\textbf{Proof of Proposition \ref{t3.1}}:\
First we have
\begin{align*}
\d (X_t-X_t^\veps)&=(\mu(X_t)-\mu_\veps(X_t^\veps))\d t
                  =(\mu(X_t)-\mu_\veps(X_t)+\mu_\veps(X_t)-\mu_{\veps}(X_t^\veps))\d t.
\end{align*}
By (H3) and (H4),
$$|X_t-X_t^\veps|\leq \veps t+\int_0^tK_2|X_s-X_s^\veps|\d s$$
and it follows from Gronwall's lemma that
\[ |X_t-X_t^\veps|\leq \veps t e^{K_2t}.\]
Then
\begin{align}\label{3.2}
\notag
&\p\big(\sup_{0\leq t\leq T}(X_t-c)<0\big)-\p\big(\sup_{0\leq t\leq T}(X_t^\veps-c)<0\big)\\ \notag
&\leq \p\big(\sup_{0\leq t\leq T}(X_t^\veps-\veps te^{K_2t}-c)<0\big)-\p\big(\sup_{0\leq t\leq T}(X_t^\veps-c)<0\big)\\ \notag
&\leq \p\big(0\leq \sup_{0\leq t\leq T}(X_t^\veps-c)<\veps Te^{K_2T}\big)=\p\big(\tau_c^\veps\leq T, \sup_{\tau_c^\veps\leq t\leq T}(X_t^\veps-c)<\veps T e^{K_2 T}\big)\\
&=\int_0^T\p(\tau_c^\veps\in \d s)\p\big(\sup_{s\leq t\leq T}(X_t^\veps-c)<\veps T e^{K_2T}\big|X_s^\veps =c\big).
\end{align}
Set $\mu_{\veps}^-=\inf\{\mu_\veps(y)\}$. By (H4), it holds that $\mu_{\veps}^{-}\geq \mu_{l}-\veps>-\infty$.
Since
$$X_t^\veps-X_s^\veps\geq \mu_{\veps}^{-}(t-s)+W_t-W_s,\quad t>s,$$
we get
\begin{equation}\label{3.3}
\begin{split}
&\p\big(\sup_{s\leq t\leq T}(X_t^\veps-c)<\veps Te^{K_2T}\big|X_s^\veps=c\big)\\
&\leq \p\big(\sup_{0\leq t\leq T-s}(\mu_{\veps}^{-}\,t +W_t)<\veps T e^{K_2T}\big)\\
&=\Phi\big(\frac{-\mu_{\veps}^{-}(T-s)+\veps Te^{K_2T}}{\sqrt{T-s}}\big)-e^{2\mu_{\veps}^{-}\veps Te^{K_2T}}\Phi\big(\frac{-\mu_{\veps}^{-}(T-s)-\veps T
e^{K_2T}}{\sqrt{T-s}}\big)\\
&=:I(T,s,K_2,\mu_{\veps}^{-};\veps).
\end{split}
\end{equation}
Substituting the above inequality into (\ref{3.2}) and using Theorem \ref{t2.3}, we obtain
\begin{equation}\label{3.4}
\begin{split}
&\p\big(\sup_{0\leq t\leq T}(X_t-c)<0\big)-\p\big(\sup_{0\leq t\leq T}(X_t^\veps-c)<0\big)\\
&\leq \int_0^T\frac{c-x}{\sqrt{2\pi} s^{3/2}}e^{G_\veps(c)-G_{\veps}(x)-\frac{3M_\veps s}{2}}e^{-\frac{(c-x)^2}{2s}}I(T,s,K_2,\mu_{\veps}^{-};\veps)\d s,
\end{split}
\end{equation}
where $G_\veps(y)=\int_{y_0}^y\mu_\veps(z)\d z$, $M_\veps=\inf\{\mu_\veps^2+\frac 13\mu_{\veps,-}'(y)\}$, and
$$\mu_{\veps,-}'(y):=\liminf_{z\ra y}\frac{\mu_\veps(z)-\mu_{\veps}(y)}{z-y}\leq K_2.$$
Further, by  $1-e^{-x}\leq |x|$ for all $x\in \R$, we have the following upper bound
\begin{equation}\label{est I}
I(T,s, K_2,\mu_{\veps}^{-};\veps)
\leq \Big(\frac{2Te^{K_2T}}{\sqrt{2\pi(T-s)}} +2(|\mu_{l}|+\veps)Te^{K_2T}\Big)\veps.
\end{equation}
Inserting \eqref{est I} into \eqref{3.4} and noting $G_\veps(x) \sim G(x)$ as $\veps \ra 0$, we get
\begin{equation*}
\begin{split}
& \p(\sup_{0\leq t\leq T} (X_t -c)<0)-\p(\sup_{0\leq t\leq T} (X_t -c)<0) \\
&\leq 2Te^{3K_2T/2}e^{G(c)-G(x)}\Big[\int_0^T\!\!\! \frac{c\!-\!x}{\sqrt{2\pi} s^{3/2}}
e^{-\frac{(c-x)^2}{2s}}\Big(|\mu_{l}|\!+\!1/\big(2\pi(T-s)\big)^{\frac 12}\Big)\d s\Big]\veps+o(\veps),
\end{split}
\end{equation*}
In order to prove the inverse direction of the inequality, we write
\begin{align*}
&\p\big(\sup_{0\leq t\leq T} (X_t^\veps-c)<0\big)-\p\big(\sup_{0\leq t\leq T}(X_t-c)<0\big)\\
&\leq \p(\sup_{0\leq t\leq T}(X_t^\veps-c)<0\big)-\p\big(\sup_{0\leq t\leq T} (X_t^\veps+\veps Te^{K_2T}-c)<0\big)\\
&\leq \p(\sup_{0\leq t\leq T}(X_t^\veps+\veps Te^{K_2T}-c)<\veps Te^{K_2T}\big)-\p\big(\sup_{0\leq t\leq T}(X_t^\veps +Te^{K_2T}-c)<0\big)\\
&=\int_0^T\p(\tau_{\tilde c}^\veps\in \d s)\p\big(\sup_{s\leq t\leq T}(X_t^\veps-\tilde c)<\veps T e^{K_2T}\big|X_s^\veps=\tilde c\big),
\end{align*}
where $\veps$ is small enough so that $c-\veps Te^{K_2T}>x$ as $c>x$, and $\tilde c=c-\veps Te^{K_2T}$.
Using the same arguments for (\ref{3.3}) and (\ref{3.4}), we can obtain the lower bound. Combining this with (\ref{est I}) and noting that
$M_\veps\geq -K_2/3$, we have
\begin{equation*}
\begin{split}
&\big|\p(\sup_{0\leq t\leq T} (X_t -c)<0)-\p(\sup_{0\leq t\leq T} (X_t -c)<0)\big|\\
&\leq 2Te^{3K_2T/2}e^{G(c)-G(x)}\Big[\int_0^T\!\!\! \frac{c\!-\!x}{\sqrt{2\pi} s^{3/2}}
e^{-\frac{(c-x)^2}{2s}}\Big(|\mu_{l}|\!+\!1/\big(2\pi(T-s)\big)^{\frac 12}\Big)\d s\Big]\veps+o(\veps),
\end{split}
\end{equation*}
which concludes the proof.

\bigskip

\noindent \textbf{Acknowledgement:} The authors are grateful to Professor James C. Fu for valuable discussions during the first author's visit to the
Department of Statistics, University of Manitoba, where the main part of this work is done. The research is partially supported by the Natural Sciences and Engineering Research Council of Canada (NSERC) and NSFC (No.11301030), 985-project.

\end{document}